\documentclass[11pt]{amsart}

\usepackage{amsfonts, fancyhdr, bbm, qtree, amsmath, amscd, tipa, amssymb, hyperref, url, ifsym, amsthm, subfigure, enumitem, wasysym, tikz, mathtools}
\usetikzlibrary{matrix, arrows}

\newtheorem{thm}{Theorem}[section]
\newtheorem{lem}[thm]{Lemma}

\newtheorem{prop}[thm]{Proposition}

\newtheorem{exam}[thm]{Example}
\theoremstyle{definition}
\newtheorem{df}[thm]{Definition}
\theoremstyle{remark}
\newtheorem{rem}[thm]{Remark}
\theoremstyle{question}

\numberwithin{equation}{section}

\title{Restriction of Global Bases and Rhoades's Theorem}
\author{David B Rush}
\email{dbr@alum.mit.edu}
\date{\today}

\begin{document}

\begin{abstract}
	It is shown that if $\lambda$ is a multiple of a fundamental weight of $\mathfrak{sl}_k$, the lower global basis of the irreducible $U_q(\mathfrak{sl}_k)$-representation $V^{\lambda}$ with highest weight $\lambda$ comprises the disjoint union of the lower global bases of the irreducible $U_q(\mathfrak{sl}_{k-1})$-representations appearing in the decomposition of the restriction of $V^{\lambda}$ to $U_q(\mathfrak{sl}_{k-1})$.  
	
	Rhoades's description of the action of the long cycle on the dual canonical basis of $V^{\lambda}$ is then deduced from Berenstein--Zelevinsky's description of the action of the long element.  This yields a short proof of Rhoades's result on tableaux fixed under promotion which directly relates it to Stembridge's result on tableaux fixed under evacuation.

\end{abstract}

\maketitle

\section{Introduction}

\subsection{Motivation}

Let $\Lambda$ be a rectangular partition.  Perhaps the most celebrated example of the \textit{cyclic sieving phenomenon} of Reiner, Stanton, and White \cite{Reiner} is the theorem of Rhoades \cite{Rhoades} that \textit{jeu-de-taquin} promotion exhibits cyclic sieving:
\begin{itemize}
	\item on the set $SYT(\Lambda)$ of standard tableaux of shape $\Lambda$ with respect to the $q$-analogue of the Weyl dimension formula for the $(1, \ldots, 1)$-weight space of the irreducible $GL_{|\Lambda|}(\mathbb{C})$-representation with highest weight $\Lambda$;
	\item on the set $SSYT(\Lambda, k)$ of semistandard tableaux of shape $\Lambda$ with entries in $\lbrace 1, \ldots, k \rbrace$ (for any positive integer $k \geq \ell(\Lambda)$) with respect to the $q$-analogue of the Weyl dimension formula for the irreducible $GL_k(\mathbb{C})$-representation with highest weight $\Lambda$.
\end{itemize}

Rhoades's theorem was featured in several surveys of the cyclic sieving phenomenon (cf. Reiner--Stanton--White \cite{Reiner2} and Sagan \cite{Sagan}) and inspired activity in and around algebraic combinatorics.  The result for standard tableaux has been reproved thrice,\footnote{Purbhoo \cite{Purbhoo} used the Wronski map; Fontaine and Kamnitzer \cite{Fontaine} and Westbury \cite {Westbury} both used invariants in tensor products of representations of semisimple Lie algebras, and in type $A$ their results reduced to Rhoades's.  Fontaine and Kamnitzer recovered Rhoades's result that promotion exhibits cyclic sieving on the set of semistandard tableaux with fixed (cyclically symmetric) content, which encompasses the set of standard tableaux.  } but the result for semistandard tableaux has only been reproved by Shen and Weng \cite{Shen}, who discovered an equivalent result via cluster algebras --- although they did not realize their result was the same (cf. Hopkins \cite{Hopkins} for the details).  None of the subsequent proofs is simpler than Rhoades's original proof adhering to the representation-theoretic paradigm for observing cyclic sieving phenomena.  

Rhoades proved that, up to sign, the long cycle $c_{|\Lambda|} \in \mathfrak{S}_{|\Lambda|}$ permutes by promotion the elements --- indexed by $SYT(\Lambda)$ --- of the \textit{Kazhdan--Lusztig basis} of the irreducible $\mathfrak{S}_{|\Lambda|}$-representation $S^{\Lambda}$ corresponding to $\Lambda$.  Then, he proved that, for any positive integer $k \geq \ell(\Lambda)$, up to sign, $c_k \in \mathfrak{S}_k$ permutes by promotion the elements --- indexed by $SSYT(\Lambda, k)$ --- of the basis (constructed by Du \cite{Du} and Skandera \cite{Skandera}) of Kazhdan--Lusztig immanants of (the dual of) the irreducible $GL_k(\mathbb{C})$-representation $V^{\Lambda}$ with highest weight $\Lambda$.  After evaluating characters, he obtained his enumerative results.  

Fourteen years before Rhoades's theorem, Stembridge \cite{Stembridge} discovered that evacuation exhibits the \textit{$q=-1$ phenomenon} (the antecedent, pertaining to involutions, of the cyclic sieving phenomenon) on the same tableaux sets, with respect to the same generating functions, that Rhoades would consider --- without the stipulation that $\Lambda$ be rectangular.  His proof followed the same framework that Rhoades's later would: He proved that, up to sign, the long element $w_{0, |\Lambda|} \in \mathfrak{S}_{|\Lambda|}$ permutes by evacuation the Kazhdan--Lusztig basis of $S^{\Lambda}$, and he quoted the result of Berenstein and Zelevinsky \cite{Berenstein} that, for any positive integer $k \geq \ell(\Lambda)$, up to sign, $w_{0,k} \in \mathfrak{S}_k$ permutes by evacuation the \textit{dual canonical basis} of $V^{\Lambda}$.\footnote{Let $\lambda$ be the image of $\Lambda$ in the weight lattice of $\mathfrak{sl}_k$.  Then $V^{\Lambda}$ is (isomorphic to) the $\mathbb{C}$-form of the $q=1$ specialization of the $U_q(\mathfrak{sl}_k)$-representation $V^{\lambda}$ with highest weight $\lambda$.  By the dual canonical basis of $V^{\Lambda}$, we mean the basis of $V^{\Lambda}$ corresponding to Lusztig's dual canonical basis of $V^{\lambda}$ (cf. Berenstein--Zelevinsky \cite{Berenstein}).  }  

The author first noted this resemblance in Rush \cite{Rush}, in which he devised a combinatorial expression for certain plethysm coefficients by refining Rhoades's and Stembridge's results to tableaux he construed to be highest weight.  That article treated the stories for promotion and evacuation as separate but parallel.  The purpose of this article is to merge the stories and provide a proof of Rhoades's theorem that elucidates its relationship to Stembridge's theorem.  

\subsection{The Kazhdan--Lusztig basis}

Rhoades did not miss the apparent likeness between his work and Stembridge's; on the contrary, he explicitly noted that his description of the action of $c_{|\Lambda|}$ on the Kazhdan--Lusztig basis of $S^{\Lambda}$ was analogous to Stembridge's description of the action of $w_{0, |\Lambda|}$.  However, he did miss that the former is actually a consequence of the latter.  

For all positive integers $i$, let $\xi_i$ denote the $i^{\text{th}}$ \textit{partial Sch\"utzenberger involution}, which acts on a tableau by performing evacuation on the subtableau containing all entries less than or equal to $i$, and let $J$ denote \textit{jeu-de-taquin} promotion.  Then, as actions on $SYT(\Lambda)$, 
\[\xi_{|\Lambda|} \circ \xi_{|\Lambda|-1} = J.\]

Furthermore, 
\[w_{0,|\Lambda|} \circ w_{0,|\Lambda|-1} = c_{|\Lambda|}.\]  
Thus, on the Kazhdan--Lusztig basis of $S^{\Lambda}$, up to sign, $c_{|\Lambda|}$ acts by $J$ if $w_{0, |\Lambda|-1}$ acts by $\xi_{|\Lambda|-1}$.  

Suppose again that $\Lambda$ is rectangular, and let $\widehat{\Lambda}$ be the partition obtained from $\Lambda$ by removing the sole outside corner.  Then $S^{\widehat{\Lambda}}$ and $S^{\Lambda}$ are isomorphic as $\mathfrak{S}_{|\Lambda|-1}$-representations, and, up to sign, $w_{0, |\Lambda|-1}$ acts by $\xi_{|\Lambda|-1}$ on the Kazhdan--Lusztig basis of $S^{\widehat{\Lambda}}$.  Therefore, it suffices to show that the Kazhdan--Lusztig bases of $S^{\widehat{\Lambda}}$ and $S^{\Lambda}$ are compatible.   

Let $a$ and $m$ be positive integers for which $\Lambda = (m^a)$.  For a standard tableau $P$ of shape $\Lambda$, let $\widehat{P}$ be the standard tableau of shape $\widehat{\Lambda}$ obtained by removing the sole outside corner.  Let $Q$ be the standard tableau of shape $\Lambda$ such that the $i^{\textit{th}}$ column of $Q$ contains the entries $ia-a+1, \ldots, ia$ for all $1 \leq i \leq m$.  Let $\mathcal{C}$ be the left cell of $\mathfrak{S}_{ma}$ comprising the permutations with recording tableau $Q$, and let $\widehat{\mathcal{C}}$ be the left cell of $\mathfrak{S}_{ma-1}$ comprising the permutations with recording tableau $\widehat{Q}$.\footnote{The reader following along with Rhoades \cite{Rhoades} should be aware that Rhoades incorrectly identifies a Kazhdan--Lusztig left cell with a fixed insertion tableau instead of a fixed recording tableau.  }  Then the map $\phi \colon \widehat{\mathcal{C}} \rightarrow \mathcal{C}$ given by 
\[\widehat{x} \mapsto \widehat{x} (ma \ ma-1 \ \cdots \ ma - a + 1)\]
is a bijection, and $\mu[\phi(\widehat{x}), \phi(\widehat{y})] = \mu[\widehat{x}, \widehat{y}]$ for all $\widehat{x}, \widehat{y} \in \widehat{\mathcal{C}}$ (cf. Fact 11 in Garsia--McLarnan \cite{Garsia}).  Therefore, the map $\tilde{\phi} \colon S^{\widehat{\Lambda}} \rightarrow S^{\Lambda}$ given by $C_{\widehat{x}} \mapsto C_{\phi(\widehat{x})}$ (where $C_{\widehat{x}} \in \mathbb{C}[\mathfrak{S}_{ma-1}]$ and $C_{\phi(\widehat{x})} \in \mathbb{C}[\mathfrak{S}_{ma}]$ are the Kazhdan--Lusztig basis elements associated to $\widehat{x}$ and $\phi(\widehat{x})$, respectively) is an isomorphism of $\mathfrak{S}_{ma-1}$-representations.  Identifying the permutations in $\widehat{\mathcal{C}}$ and $\mathcal{C}$ with their respective insertion tableaux, we see that $\tilde{\phi}^{-1}$ is given by $C_{P} \mapsto C_{\widehat{P}}$.  

Hence, by Theorem 5.1 of Stembridge \cite{Stembridge}, there exist $\epsilon_{\Lambda}, \epsilon_{\widehat{\Lambda}} \in \lbrace \pm 1 \rbrace$ such that for all $P \in SYT(\Lambda)$,
\begin{align*}
c_{|\Lambda|} \cdot C_{P} & = w_{0,|\Lambda|} w_{0, |\Lambda|-1} \cdot C_{P} = w_{0, |\Lambda|} \cdot \tilde{\phi}\left(w_{0, |\Lambda|-1} \cdot C_{\widehat{P}}\right) 
\\ & =  w_{0, |\Lambda|} \cdot \tilde{\phi}\left(\epsilon_{\widehat{\Lambda}} C_{\xi_{|\Lambda|-1} (\widehat{P})}\right) =  w_{0, |\Lambda|} \cdot \epsilon_{\widehat{\Lambda}} C_{\xi_{|\Lambda|-1} (P)} 
\\ & = \epsilon_{\Lambda} \epsilon_{\widehat{\Lambda}} C_{\xi_{|\Lambda|}( \xi_{|\Lambda|-1} (P))} = \epsilon_{\Lambda} \epsilon_{\widehat{\Lambda}} C_{J(P)},
\end{align*}
whence Rhoades's result for standard tableaux follows.  

\subsection{The dual canonical basis}

Since Rhoades obtains his description of the action of $c_k$ on the basis of Kazhdan--Lusztig immanants of (the dual of) $V^{\Lambda}$ from his description of the action of $c_{|\Lambda|}$ on the Kazhdan--Lusztig basis of $S^{\Lambda}$, the argument we have just given is sufficient to conclude Rhoades's result for semistandard tableaux as well.  To do so, however, would be to sidestep the challenge of offering an analogous argument relating Rhoades's description of the action of $c_k$ to Berenstein and Zelevinsky's description of the action of $w_{0,k}$.  That the basis of Kazhdan--Lusztig immanants is (essentially) the dual canonical basis of $V^{\Lambda}$ (cf. Skandera \cite{Skandera}) suggests that such an argument should be feasible.  But it would entail a subtler invocation of the hypothesis that $\Lambda$ is rectangular, for, unlike the restriction of $S^{\Lambda}$ to $\mathfrak{S}_{|\Lambda|-1}$, the restriction of $V^{\lambda}$ to $U_q(\mathfrak{sl}_{k-1})$ is not irreducible.  

To this challenge we devote the remainder of the article.  We depart from the approach of Berenstein--Zelevinsky \cite{Berenstein} and Rhoades \cite{Rhoades}, who rely on explicit constructions of the dual canonical basis (the better to analyze combinatorial actions on the basis elements), and turn to a characterization of canonical bases due to Kashiwara \cite{Kashiwara2}.  In particular, we study the restriction of Kashiwara's \textit{lower} and \textit{upper global bases}, which coincide with Lusztig's canonical and dual canonical bases, respectively (cf. Grojnowski--Lusztig \cite{Grojnowski}).  We find that the decomposition of $V^{\lambda}$ as a direct sum of irreducible $U_q(\mathfrak{sl}_{k-1})$-representations respects the global bases: The lower global basis of $V^{\lambda}$ is the disjoint union of the lower global bases of the representations appearing in the decomposition, and the same holds for the upper global basis (up to scaling by Gaussian polynomials in $q$), just as we would wish.  

We provide background on quantum groups and crystal bases in section 2 and discuss the crystal structure on tableaux in section 3.  We carry out our investigation of global bases in section 4.  Then, in section 5, we realize Rhoades's description of the action of $c_k$ as a consequence of Berenstein and Zelevinky's description of the action of $w_{0,k}$ --- and thereby arrive at a direct proof of Rhoades's theorem.  

\section{Background}

\subsection{Quantum groups}

Let $\mathfrak{h} \subset \mathfrak{sl}_k(\mathbb{C})$ be a Cartan subalgebra, and let $P^{\vee} \subset \mathfrak{h}$ and $P \subset \mathfrak{h}^*$ be the coroot and weight lattices, respectively.  

Identify $P$ with the $\mathbb{Z}$-module generated by the symbols $E_1, \ldots, E_k$ subject to the relation $E_1 + \cdots + E_k = 0$.  For all $1 \leq i \leq k-1$, set $\alpha_i := E_i - E_{i+1}$.  

Choose $\alpha_1, \ldots, \alpha_{k-1} \in P$ to be the simple roots.  Let $h_1, \ldots, h_{k-1} \in P^{\vee}$ and $\omega_1, \ldots, \omega_{k-1} \in P$ be the corresponding simple coroots and fundamental weights, respectively.  Note that $h_1, \ldots, h_{k-1}$ generate $P^{\vee}$ and $\omega_1, \ldots, \omega_{k-1}$ generate $P$.    

A weight $\lambda \in P$ is \textit{dominant} if it can be expressed as a nonnegative linear combination of fundamental weights.  Let $P^+ \subset P$ be the semigroup of dominant weights.  

\begin{df}
The quantum group $U_q(\mathfrak{sl}_k)$ is the (unital) associative algebra over $\mathbb{Q}(q)$ generated by the symbols $e_i, f_i \enspace (1 \leq i \leq k-1)$ and $q^h \enspace (h \in P^{\vee})$ subject to the following relations:
\begin{enumerate}
	\item $q^0 = 1$ and $q^h q^{h'} = q^{h+h'}$ for $h, h' \in P^{\vee}$;
	\item $q^h e_i q^{-h} = q^{\alpha_i(h)} e_i$ and $q^h f_i q^{-h} = q^{-\alpha_i(h)} f_i$ for $h \in P^{\vee}$;
	\item $e_i f_j - f_j e_i = \delta_{i,j} \frac{q^{h_i}-q^{-h_i}}{q - q^{-1}}$;
	\item $e_i^2 e_j - (q + q^{-1}) e_i e_j e_i + e_j e_i^2 = f_i^2 f_j - (q + q^{-1}) f_i f_j f_i + f_j f_i^2 = 0$ for $|i-j|=1$;
	\item $e_i e_j - e_j e_i = f_i f_j - f_j f_i = 0$ for $|i-j|>1$.  
\end{enumerate}
\end{df}

For all $n \in \mathbb{Z}$, set $[x; n]_q := \frac{x q^n - x^{-1}q^{-n}}{q-q^{-1}}$ and $[n]_q := [1;n]_q$.  For all $n \geq 0$, set $e_i^{(n)} := \frac{e_i^n}{[n]_q!}$ and $f_i^{(n)} := \frac{f_i^n}{[n]_q!}$, where $[n]_q! := [n]_q \cdots [1]_q$ for all $n \geq 1$ and $[0]_q! := 1$.  

\subsection{Representations and crystal bases}

In this article, by a $U_q(\mathfrak{sl}_k)$-representation we mean a $U_q(\mathfrak{sl}_k)$-module $V$, finite-dimensional as a vector space over $\mathbb{Q}(q)$, admitting a \textit{weight space decomposition} $V = \bigoplus_{\mu \in P} V_{\mu}$, where 
\[V_{\mu} := \lbrace v \in V : q^h v = q^{\mu(h)}v \quad \forall h \in P^{\vee} \rbrace.\]

Given a $U_q(\mathfrak{sl}_k)$-representation $V$, we say that $\mu \in P$ is a \textit{weight} of $V$ if the $\mu$-weight space $V_{\mu}$ is nonzero, in which case we refer to the nonzero vectors in $V_{\mu}$ as \textit{weight vectors}.  If $\lambda$ is a weight of $V$ and there exists a weight vector $v_{\lambda} \in V_{\lambda}$ such that $e_i v_{\lambda} = 0$ for all $1 \leq i \leq k-1$ and $V = U_q(\mathfrak{sl}_k) v_{\lambda}$, we say that $\lambda$ is the \textit{highest weight} of $V$.  

\begin{prop}[Hong--Kang \cite{Hong}, Corollary 3.4.8]
	For all $\lambda \in P^+$, there exists an irreducible $U_q(\mathfrak{sl}_k)$-representation $V^{\lambda}$ with highest weight $\lambda$.  Furthermore, the map $\lambda \mapsto [V^{\lambda}]$ defines a bijection between $P^+$ and the set of isomorphism classes of irreducible $U_q(\mathfrak{sl}_k)$-representations.  
\end{prop}

To define the global bases of $V^{\lambda}$, we must first define a crystal basis, for which we require Kashiwara's $\mathbb{Q}(q)$-linear operators $\tilde{e}_i$ and $\tilde{f}_i$.  

\begin{df}
	Let $V$ be a $U_q(\mathfrak{sl}_k)$-representation, and let $\mu$ be a weight of $V$.  Each weight vector $u \in V_{\mu}$ uniquely determines a nonnegative integer $N$ and weight vectors $u_n \in V_{\mu + n \alpha_i} \cap \ker e_i$ for all $0 \leq n \leq N$ such that $u = u_0 + f_i u_1 + \cdots + f_i^{(N)} u_N$ (cf. Lemma 4.1.1 in Hong--Kang \cite{Hong}).  Then 
	\[\tilde{e}_i u := \sum_{n=1}^N f_i^{(n-1)} u_n \quad \text{and} \quad \tilde{f}_i u := \sum_{n=0}^N f_i^{(n+1)} u_n.\] 
\end{df}

Let $A_0$ be the localization of the ring $\mathbb{Q}[q]$ at the prime ideal $(q)$.  

\begin{df}
	Let $V$ be a $U_q(\mathfrak{sl}_k)$-representation.  A free $A_0$-submodule $L \subset V$ such that $\mathbb{Q}(q) \otimes_{A_0} L \cong V$ is a \textit{crystal lattice} if:
	\begin{enumerate}
		\item $L = \bigoplus_{\mu \in P} L_{\mu}$, where $L_{\mu} := L \cap V_{\mu}$;
		\item $\tilde{e}_i L \subset L$ and $\tilde{f}_i L \subset L$ for all $1 \leq i \leq k-1$.  
	\end{enumerate}
\end{df}

Given a crystal lattice $L$, it follows from condition (2) that the operators $\tilde{e}_i$ and $\tilde{f}_i$ act on the quotient $L/qL$.  For all $v \in L$, we denote by $\overline{v}$ the image of $v$ in $L/qL$.  

\begin{df}
	Let $V$ be a $U_q(\mathfrak{sl}_k)$-representation.  A pair $(L, B)$ consisting of a crystal lattice $L$ of $V$ and a $\mathbb{Q}$-basis $B$ of $L/qL$ is a \textit{crystal basis} if:
	\begin{enumerate}
		\item $B = \bigsqcup_{\mu \in P} B_{\mu}$, where $B_{\mu} := B \cap L_{\mu}/qL_{\mu}$;
		\item $\tilde{e}_i B \subset B \cup \lbrace 0 \rbrace$ and $\tilde{f}_i B \subset B \cup \lbrace 0 \rbrace$ for all $1 \leq i \leq k-1$;
		\item $\tilde{f}_i b = b'$ if and only if $\tilde{e}_i b' = b$ for all $b, b' \in B$.
	\end{enumerate}
\end{df}

Let $\lambda \in P^+$, and fix a weight vector $v_{\lambda} \in V^{\lambda}_{\lambda}$.  

\begin{prop}[Hong--Kang \cite{Hong}, Theorem 5.1.1; Kashiwara \cite{Kashiwara}, Theorem 4.2] \label{lattice}
	There exists a unique crystal basis $(L^{\lambda}, B^{\lambda})$ of $V^{\lambda}$ such that $L^{\lambda}_{\lambda} = A_0 v_{\lambda}$ and $B^{\lambda}_{\lambda} = \lbrace \overline{v_{\lambda}} \rbrace$.  Furthermore, 
	\[L^{\lambda} = \operatorname{span} \lbrace \tilde{f}_{i_r} \cdots \tilde{f}_{i_1} v_{\lambda} : r \geq 0; 1 \leq i_1, \ldots, i_r \leq k-1 \rbrace,\] and
	\[B^{\lambda} = \lbrace \tilde{f}_{i_r} \cdots \tilde{f}_{i_1} \overline{v_{\lambda}} : r \geq 0; 1 \leq i_1, \ldots, i_r \leq k-1 \rbrace \setminus \lbrace 0 \rbrace.\]
\end{prop}

\subsection{Global bases}

Set $A := \mathbb{Q}[q, q^{-1}]$.  Let $U_A(\mathfrak{sl}_{k}) \subset U_q(\mathfrak{sl}_k)$ be the $A$-subalgebra generated by $e_i^{(n)}, f_i^{(n)}, q^h, \frac{[q^h;0]_q \cdots [q^h;1-n]_q}{[n]_q!} \enspace (n \geq 1)$.  

Let $\psi \colon U_q(\mathfrak{sl}_{k}) \rightarrow U_q(\mathfrak{sl}_{k})$ be the $\mathbb{Q}$-algebra automorphism given by 
\[e_i \mapsto e_i, \quad f_i \mapsto f_i, \quad q \mapsto q^{-1}, \quad \text{and} \quad q^h \mapsto q^{-h},\] 
and let $\psi$ also denote the induced $\mathbb{Q}$-linear automorphism of $V^{\lambda}$ given by $p v_{\lambda} \mapsto \psi(p) v_{\lambda}$ for all $p \in U_q(\mathfrak{sl}_k)$.  

\begin{thm}[Hong--Kang \cite{Hong}, Section 6.2] \label{global}
	There exists a unique $A_0$-basis $G^{\lambda} = \lbrace G_b \rbrace_{b \in B^{\lambda}}$ of $L^{\lambda}$ such that (i) $G^{\lambda}$ is a $\mathbb{Q}$-basis of $U_A(\mathfrak{sl}_k) v_{\lambda} \cap L^{\lambda} \cap \psi(L^{\lambda})$, and (ii) $\overline{G_b} = b$ and $\psi(G_b) = G_b$ for all $b \in B^{\lambda}$.    
\end{thm}

Let $\varphi \colon U_q(\mathfrak{sl}_k) \rightarrow U_q(\mathfrak{sl}_k)$ be the $\mathbb{Q}$-algebra antiautomorphism given by 
\[e_i \mapsto f_i, \quad f_i \mapsto e_i, \quad q \mapsto q, \quad \text{and} \quad q^h \mapsto q^h.\]  The \textit{Shapovalov form} on $V^{\lambda}$ is the unique symmetric bilinear form $(\cdot, \cdot)$ such that $(v_{\lambda}, v_{\lambda}) = 1$ and $(pu, v) = (u, \varphi(p)v)$ for all $u, v \in V^{\lambda}$ and $p \in U_q(\mathfrak{sl}_k)$.  

\begin{prop} \label{perpen}
	If $\mu, \nu \in P$ are distinct weights of $V^{\lambda}$, then $(V^{\lambda}_{\mu}, V^{\lambda}_{\nu}) = 0$.  
\end{prop}

\begin{proof}
	Let $u \in V^{\lambda}_{\mu}$ and $v \in V^{\lambda}_{\nu}$.  For all $h \in P^{\vee}$, \[q^{\mu(h)} (u,v) = (q^h u, v) = (u, q^h v) = q^{\nu(h)} (u, v).\]
\end{proof}

\begin{df}
	The \textit{lower global basis} of $V^{\lambda}$ is $G^{\lambda}$, and the \textit{upper global basis} $F^{\lambda}$ of $V^{\lambda}$ is the dual basis to $G^{\lambda}$ with respect to the Shapovalov form.  
\end{df}

\begin{prop} \label{wvec}
	$G^{\lambda}_{\mu} := \lbrace G_b \rbrace_{b \in B^{\lambda}_{\mu}}$ and $F^{\lambda}_{\mu} := \lbrace F_b \rbrace_{b \in B^{\lambda}_{\mu}}$ are $\mathbb{Q}(q)$-bases of $V^{\lambda}_{\mu}$ for all $\mu \in P$.  
\end{prop}

\begin{proof}
	From Lusztig's construction of the canonical basis (cf. Lusztig \cite{Lusztig}), we see that $G^{\lambda} \cap V^{\lambda}_{\mu}$ is a $\mathbb{Q}(q)$-basis of $V^{\lambda}_{\mu}$ for all $\mu \in P$.  Hence $G^{\lambda}$ consists of weight vectors, so $G^{\lambda}_{\mu} \subset V^{\lambda}_{\mu}$ for all $\mu \in P$, which implies $G^{\lambda}_{\mu} = G^{\lambda} \cap V^{\lambda}_{\mu}$ for all $\mu \in P$.  
	
	It follows that $(F^{\lambda}_{\mu}, V^{\lambda}_{\nu}) = 0$ for all distinct $\mu, \nu \in P$ (cf. Proposition~\ref{perpen}).  The Shapovalov form is nondegenerate, so $F^{\lambda}_{\mu} \subset V^{\lambda}_{\mu}$ for all $\mu \in P$.  
\end{proof}

\subsection{Classical limit}

Let $\mathsf{V}^{\lambda}$ be the $\mathbb{C}$-vector space generated by the symbols $\mathsf{F}_b$ for $b \in B_{\lambda}$.  For all $i$, the entries of the matrices of $e_i$ and $f_i$ on the $\mathbb{Q}(q)$-vector space $V^{\lambda} = \bigoplus_{b \in B^{\lambda}} \mathbb{Q}(q) F_b$ are regular at $q=1$ (cf. Berenstein--Zelevinsky \cite{Berenstein}).  

Thus, specializing the matrices of $e_i$ and $f_i$ at $q = 1$ for all $i$, and letting $h$ act as multiplication by $\mu(h)$ on $\mathsf{F}^{\lambda}_{\mu} := \lbrace \mathsf{F}_b \rbrace_{b \in B^{\lambda}_{\mu}}$ for all $\mu \in P$ and $h \in P^{\vee}$ equips the $\mathbb{C}$-vector space $\mathsf{V}^{\lambda} = \bigoplus_{b \in B^{\lambda}} \mathbb{C} \mathsf{F}_b$ with the structure of an $\mathfrak{sl}_k$-module.  

As an $\mathfrak{sl}_k$-representation, $\mathsf{V}^{\lambda}$ is irreducible with highest weight $\lambda$ (cf. \cite{Berenstein}).  We refer to $\mathsf{V}^{\lambda}$ as the $\mathbb{C}$-form of the $q=1$ specialization of $V^{\lambda}$.  

\section{Crystal structure on tableaux}

\subsection{Review}

Let $\lambda \in P^+$, and let $l_1, \ldots, l_{k-1}$ be nonnegative integers such that $\lambda = l_1 \omega_1 + \cdots + l_{k-1} \omega_{k-1}$.  For all $1 \leq i \leq k-1$, set $\lambda_i := l_i + \cdots + l_{k-1}$, and let $\lambda$ also denote the partition $(\lambda_1, \ldots, \lambda_{k-1})$.  

The crystal basis $(L^{\lambda}, B^{\lambda})$ is an effective combinatorial model for the $U_q(\mathfrak{sl}_k)$-representation $V^{\lambda}$ in part because there exists an identification of $B^{\lambda}$ with $SSYT(\lambda, k)$ under which the action of Kashiwara's operators admits a simple description.  

\begin{df}
	Label the boxes in the Young diagram of shape $\lambda$ with the integers in $\lbrace 1, \ldots, |\lambda| \rbrace$ so that, for all $1 \leq i \leq k-1$, the boxes in the $i^{\text{th}}$ row are assigned the labels $\lambda_1 + \cdots + \lambda_{i-1} + 1, \ldots, \lambda_1 + \cdots + \lambda_i$, and the labels increase \textit{from right to left} within each row.  
	  
	Let $T \in SSYT(\lambda, k)$, and let $w_T \colon \lbrace 1, \ldots, |\lambda| \rbrace \rightarrow \lbrace 1, \ldots, k \rbrace$ be the map associating to each label the entry in the underlying box of $T$.  From the set $\lbrace 1, \ldots, |\lambda| \rbrace$, remove all $j$ for which $w_T(j) \notin \lbrace i, i+1 \rbrace$.  Then, iteratively remove all consecutive $j < j'$ for which $w_T(j) = i$ and $w_T(j') = i+1$.\footnote{We understand $j < j'$ to be consecutive if either $j' = j+1$ or all integers between $j$ and $j'$ have already been removed from the set.  }    
	
	Let $S_i$ denote the remaining set, and set $w_{T,i} := w_T|_{S_i}$.  
	\begin{itemize}
		\item If $w_{T,i}^{-1}(i+1)$ is nonempty, let $\tilde{e}_i(T)$ be the tableau obtained from $T$ by changing the entry in the box labeled \enspace $\max w_{T,i}^{-1}(i+1)$ \enspace from $i+1$ to $i$; otherwise, set $\tilde{e}_i(T) :=0$.  
		\item If $w_{T,i}^{-1}(i)$ is nonempty, let $\tilde{f}_i(T)$ be the tableau obtained from $T$ by changing the entry in the box labeled \enspace $\min w_{T,i}^{-1}(i)$ \enspace from $i$ to $i+1$; otherwise, set $\tilde{f}_i(T) :=0$.  
	\end{itemize}
\end{df}

\begin{rem}
	We depart from Hong--Kang \cite{Hong} and follow Kashiwara \cite{Kashiwara} in describing the action of $\tilde{e}_i$ and $\tilde{f}_i$ on tableaux because the approach in \cite{Hong} relies on the tensor product rule for crystals, which is outside the scope of this article.  The labeling in \cite{Kashiwara} of the boxes of $\lambda$ differs from ours, but any \textit{admissible} labeling yields the same crystal structure (cf. Theorem 7.3.6 in \cite{Hong}).  The crystal structure on tableaux is originally due to Kashiwara and Nakashima \cite{KashiwaraN}.  
\end{rem}

\begin{thm}[Kashiwara \cite{Kashiwara}, Theorem 5.1] \label{crystal}
	The operators $\tilde{e}_i$ and $\tilde{f}_i$ map $SSYT(\lambda, k)$ into $SSYT(\lambda, k) \cup \lbrace 0 \rbrace$ for all $1 \leq i \leq k-1$.  
	
	Furthermore, let $T_{\lambda} \in SSYT(\lambda, k)$ be the tableau satisfying $w_{T_{\lambda}}^{-1}(i) = \lbrace \lambda_1 + \cdots + \lambda_{i-1} + 1, \ldots, \lambda_1 + \cdots + \lambda_i \rbrace$ for all $1 \leq i \leq k-1$.  Then there exists a bijection $\pi \colon B^{\lambda} \rightarrow SSYT(\lambda, k)$ such that $\pi(\overline{v_{\lambda}}) = T_{\lambda}$ and $\pi$ commutes with $\tilde{e}_i$ and $\tilde{f}_i$.     
\end{thm}

\subsection{Rectangular tableaux}

Let $a \in \lbrace 1, \ldots, k-1 \rbrace$, and let $m$ be a positive integer.  Set $\lambda := m \omega_a$.  We present two lemmas, the latter of which we appeal to in the following section.  

\begin{lem} \label{expl}
	For all $a \leq a' \leq k-1$, and all sequences of nonnegative integers $(c_a, \ldots, c_{a'})$, set \[T_{(c_a, \ldots, c_{a'})} := \tilde{f}_{a'}^{c_{a'}} \cdots \tilde{f}_a^{c_a} (T_{\lambda}).\]  
	
	Suppose $m \geq c_a \geq \cdots \geq c_{a'} \geq 0$.  Set $c_{a-1} := m$ and $c_{a'+1} := 0$.  Then $T_{(c_a, \ldots, c_{a'})} \in SSYT(\lambda, k)$, and $w_{T_{(c_a, \ldots, c_{a'})}}$ is described by the following conditions: 
	\begin{itemize}
		\item $w_{T_{(c_a, \ldots, c_{a'})}}^{-1}(i) = \lbrace (i-1)m + 1, \ldots, im \rbrace$ for all $1 \leq i \leq a-1$;
		\item $w_{T_{(c_a, \ldots, c_{a'})}}^{-1}(i) = \lbrace (a-1)m + c_i + 1, \ldots, (a-1)m + c_{i-1} \rbrace$ for all $a \leq i \leq a' + 1$;	
	\end{itemize}
\end{lem} 

\begin{proof}
	We induct on $a'$.  Consider the action of $\tilde{f}_{a'}$ on $T_{(c_a, \ldots, c_{a'-1})}$.  By the inductive hypothesis, $w_{T_{(c_a, \ldots, c_{a'-1})}}^{-1}(a') = \lbrace (a-1)m + 1, \ldots, (a-1)m + c_{a'-1} \rbrace$, and $w_{T_{(c_a, \ldots, c_{a'-1})}}(a'+1)$ is empty.  Therefore, \[S_{a'} = \lbrace (a-1)m + 1, \ldots, (a-1)m + c_{a'-1} \rbrace.\]
	
	Provided $c_{a' - 1} \geq c_{a'} \geq 0$, we see by induction on $c_{a'}$ that $\tilde{f}_{a'}^{c_{a'}}$ changes the entries of $T_{(c_a, \ldots, c_{a'-1})}$ in the boxes labeled $(a-1)m + 1, \ldots, (a-1)m + c_{a'}$ from $a'$ to $a'+1$.  
\end{proof}

\begin{lem} \label{zero}
	The following assertions hold.  
	\begin{enumerate}
		\item For all $a \leq a' \leq k-1$, and all sequences of nonnegative integers $(c_a, \ldots, c_{a'})$, 
		\[\tilde{f}_{a'}^{c_{a'}} \cdots \tilde{f}_a^{c_a} \overline{v_{\lambda}} \] is nonzero if and only if $m \geq c_a \geq \cdots \geq c_{a'} \geq 0$.  
		\item For all $a+1 \leq a' \leq k-2$, and all $0 \leq j \leq m$, 
		\[\tilde{f}_{a'} \tilde{f}_{a'+1}^{j} \cdots \tilde{f}_a^{j} \overline{v_{\lambda}} = 0.\]
		\item For all $0 \leq j \leq m$, 
		\[\tilde{f}_{a-1}^j \tilde{f}_a^j \overline{v_{\lambda}} \neq 0 \quad \text{and} \quad \tilde{f}_{a-1}^{j+1} \tilde{f}_a^j \overline{v_{\lambda}} = 0.\]
		\item For all $0 \leq j \leq m$,
		\[\tilde{f}_a^{m-j} \tilde{f}_{a+1}^{j} \tilde{f}_a^j \overline{v_{\lambda}} \neq 0 \quad \text{and} \quad \tilde{f}_a^{m-j+1} \tilde{f}_{a+1}^{j} \tilde{f}_a^j \overline{v_{\lambda}} = 0.\]
	\end{enumerate}
	
\end{lem}

\begin{proof}
	In view of Theorem~\ref{crystal}, it suffices to show each assertion with $T_{\lambda}$ in place of $\overline{v_{\lambda}}$.  We freely apply Lemma~\ref{expl} throughout.  
	\begin{enumerate}
		\item The ``if'' direction is immediate.  
		
		For the ``only if'' direction, assume the contrary.  Set $c_{a-1} := m$, and let $a'$ be minimal for which there exists a sequence $(c_a, \ldots, c_{a'})$ with $c_{a'-1} < c_{a'}$ such that $\tilde{f}_{a'}^{c_{a'}} \cdots \tilde{f}_a^{c_a} \overline{v_{\lambda}} \neq 0$.  
	
		Since $w_{T_{(c_a, \ldots, c_{a'-1}, c_{a'-1})}}^{-1}(a')$ is empty, it follows that $\tilde{f}_{a'}$ vanishes on $T_{(c_a, \ldots, c_{a'-1}, c_{a'-1})}$, so $T_{(c_a, \ldots, c_{a'-1}, c_{a'})} = 0$.  
		
		\item Immediate.  
		
		\item Since $w_{T_{(j)}}^{-1}(a-1) = \lbrace (a-2)m + 1, \ldots, (a-1)m \rbrace$ and $w_{T_{(j)}}^{-1}(a) = \lbrace (a-1)m + j + 1, \ldots, am \rbrace$, we see that $\tilde{f}_{a-1}^j$ changes the entries of $T_{(j)}$ in the boxes labeled $(a-2)m +1, \ldots, (a-2)m + j$ from $a-1$ to $a$, and $\tilde{f}_{a-1}$ annihilates the tableau so obtained.  
		
		\item Since $w_{T_{(j,j)}}^{-1}(a) = \lbrace (a-1)m + j + 1, \ldots, am \rbrace$ and $w_{T_{(j,j)}}^{-1}(a+1)$ is empty, we see that $\tilde{f}_a^{m-j}$ changes the entries of $T_{(j,j)}$ in the boxes labeled $(a-1)m + j + 1, \ldots, am$ from $a$ to $a+1$, and $\tilde{f}_a$ annihilates the tableau so obtained.  
	\end{enumerate}
\end{proof}

\section{Restriction of Global Bases}

\subsection{Restricting $V^{\lambda}$ to $U_q(\mathfrak{sl}_{k-1})$}

We start with two lemmas that assist us in ``lifting'' to $L^{\lambda}$ the assertions in Lemma~\ref{zero} concerning elements of $L^{\lambda}/qL^{\lambda}$.  These lemmas do not depend on $\lambda$ being rectangular.  

\begin{lem} \label{lift}
	Let $N$ be a positive integer, and let $v \in L^{\lambda}$ be a weight vector.  Suppose that there exists a weight vector $v_N \in L^{\lambda}$ such that $e_i v_N = 0$ and $v = f_i^{(N)} v_N$.  If $\overline{v} \neq 0$ and $\tilde{f}_i \overline{v} = 0$, then $\tilde{f}_i v = 0$.  
\end{lem}

\begin{proof}
	Assume the contrary.  Since $\tilde{f}_i v = f_i^{(N+1)} v_N$ is nonzero, it follows that $v = \tilde{e}_i \tilde{f}_i v$, whence $\overline{v} = \tilde{e}_i \tilde{f}_i \overline{v} = 0$.  
\end{proof}

\begin{lem} \label{convert}
	For all sequences of nonnegative integers $(c_1, \ldots, c_{k-1})$, 
	\[\tilde{f}_{k-1}^{c_{k-1}} \cdots \tilde{f}_1^{c_1} v_{\lambda} = f_{k-1}^{(c_{k-1})} \cdots f_1^{(c_1)} v_{\lambda}.\]
\end{lem}

\begin{proof}
	Note that $f_{i-1}^{(c_{i-1})} \cdots f_1^{(c_1)} v_{\lambda} \in \ker e_i$ for all $1 \leq i \leq k-1$.  
\end{proof}

For the remainder of the section, set $\lambda := m \omega_a$.  

\begin{prop} \label{tech}
	The following assertions hold.  
	\begin{enumerate}
		\item For all $a \leq a' \leq k-1$, and all sequences of nonnegative integers $(c_a, \ldots, c_{a'})$, \[f_{a'}^{(c_{a'})} \cdots f_a^{(c_a)} v_{\lambda} \] is nonzero if and only if $m \geq c_a \geq \cdots \geq c_{a'} \geq 0$.  
		\item For all $a \leq a' \leq k-2$, and all $0 \leq j \leq m$,
		\[e_{a'} f_{a'+1}^{(j)} \cdots f_a^{(j)} v_{\lambda} = 0.\]
		\item For all $a+1 \leq a' \leq k-2$, and all $0 \leq j \leq m$,
		\[f_{a'}f_{a'+1}^{(j)} \cdots f_a^{(j)} v_{\lambda} = 0.\]  
		\item $f_{a-1}^{(j+1)} f_{a}^{(j)} v_{\lambda} = 0$ for all $0 \leq j \leq m$.  
		\item $f_{a}^{(m-j+1)} f_{a+1}^{(j)} f_{a}^{(j)} v_{\lambda} = 0$ for all $0 \leq j \leq m$.  
	\end{enumerate}
\end{prop}

\begin{proof}
	We freely apply Lemmas~\ref{zero}, \ref{lift}, and \ref{convert} throughout.  
	\begin{enumerate}
		\item The ``if'' direction is immediate.  
		
		For the ``only if'' direction, assume the contrary.  Set $c_{a-1} := m$, and let $a'$ be minimal for which there exists a sequence $(c_a, \ldots, c_{a'})$ with $c_{a'-1} < c_{a'}$ such that $f_{a'}^{(c_{a'})} \cdots f_a^{(c_a)} v_{\lambda} \neq 0$.  
		
		Note that \[\tilde{f}_{a'}^{c_{a'-1}} \tilde{f}_{a'-1}^{c_{a'-1}} \cdots \tilde{f}_a^{c_a} \overline{v_{\lambda}} \neq 0 \quad \text{and} \quad \tilde{f}_{a'}^{c_{a'-1}+1} \tilde{f}_{a'-1}^{c_{a'-1}} \cdots \tilde{f}_a^{c_a} \overline{v_{\lambda}} = 0.\]
		
		\item For all $a \leq a' \leq k-1$, set $v_{a', j} := f_{a'}^{(j)} \cdots f_{a}^{(j)} v_{\lambda}$, and set $v_{a-1,j} := v_{\lambda}$.  From the identity $e_{a'} f_{a'}^{(j)} = f_{a'}^{(j)} e_{a'} + f_{a'}^{(j-1)} [q^{h_{a'}};1-j]_q$ (cf. Lemma 3.2.5 in Hong--Kang \cite{Hong}), we see that 
		\begin{align*}
		e_{a'} v_{a'+1,j} & = f_{a'+1}^{(j)} e_{a'} f_{a'}^{(j)} v_{a'-1,j} \\
		& = f_{a'+1}^{(j)} f_{a'}^{(j)} e_{a'} v_{a'-1,j}  + f_{a'+1}^{(j)} f_{a'}^{(j-1)} [q^{h_{a'}};1-j]_q v_{a'-1,j}.  
		\end{align*}

		The first summand is zero.  Since $[q^{h_{a'}}; 1-j]_q$ acts as a scalar in $\mathbb{Q}(q)$ on the weight vector $v_{a'-1,j}$, and $f_{a'+1}^{(j)} f_{a'}^{(j-1)} v_{a'-1,j} = 0$ by assertion (1), it follows that the second summand is also zero.  
		
		\item Since $e_{a'} v_{a'+1,j} = 0$ by assertion (2), it follows that $\tilde{f}_{a'} v_{a'+1,j}= f_{a'} v_{a'+1,j}$.  Note that $\overline{v_{a'+1,j}} \neq 0$ and $\tilde{f}_{a'} \overline{v_{a'+1,j}} = 0$.  
		
		\item Since $e_{a-1} v_{a,j} = 0$, it follows that $\tilde{f}_{a-1}^{j+1} v_{a,j} = f_{a-1}^{(j+1)} v_{a,j}$.  Note that $\tilde{f}_{a-1}^{j} \overline{v_{a,j}} \neq 0$ and $\tilde{f}_{a-1}^{j+1} \overline{v_{a,j}} = 0$.  
		
		\item Since $e_a v_{a+1,j} = 0$ by assertion (2), it follows that $\tilde{f}_a^{m-j+1} v_{a+1,j} = f_a^{(m-j+1)} v_{a+1,j}$.  Note that $\tilde{f}_a^{m-j} \overline{v_{a+1,j}} \neq 0$ and $\tilde{f}_a^{m-j+1} \overline{v_{a+1,j}} = 0$.  

	\end{enumerate}
\end{proof}

The following proposition follows immediately from Proposition~\ref{tech}.  
\begin{prop} \label{vanish}
	For all $0 \leq j \leq m$, set $v_{j} := f_{k-1}^{(j)} \cdots f_{a}^{(j)} v_{\lambda}$.  Then, for all $0 \leq j \leq m$, the following conditions hold:
	\begin{enumerate}
		\item $v_{j} \neq 0$;
		\item $v_{j} \in \ker e_1 \cap \cdots \cap \ker e_{k-2}$;
		\item $v_{j} \in \ker f_1 \cap \cdots \cap \ker f_{a-2} \cap \ker f_{a+1} \cap \cdots \cap \ker f_{k-2}$;
		\item $v_{j} \in \ker f_{a-1}^{j+1} \cap \ker f_{a}^{m-j+1}$.  
	\end{enumerate}
\end{prop}

We are finally ready to describe explicitly the decomposition of $V^{\lambda}$ as a direct sum of irreducible $U_q(\mathfrak{sl}_{k-1})$-representations.  

\begin{prop} \label{iso}
	For all $0 \leq j \leq m$, the $U_q(\mathfrak{sl}_{k-1})$-subrepresentation $U_q(\mathfrak{sl}_{k-1}) v_{j}$ is isomorphic to the irreducible $U_q(\mathfrak{sl}_{k-1})$-representation with highest weight $\lambda^j := j \omega_{a-1} + (m-j) \omega_a$.  
\end{prop}

\begin{proof}
	Note that $q^h v_j = q^{m \omega_a(h)-j\alpha_{a}(h) - \cdots - j\alpha_{k-1}(h)} v_j$ for all $h \in P^{\vee}$.  Since \[m \omega_a - j \alpha_{a} - \cdots - j \alpha_{k-1} = j \omega_{a-1} + (m-j) \omega_a + j E_k,\] 
	and the image of $E_k$ in the weight lattice of $\mathfrak{sl}_{k-1}$ is zero, $v_j$ belongs to the $\lambda^j$-weight space of the restriction of $V^{\lambda}$ to $U_q(\mathfrak{sl}_{k-1})$.  In view of Proposition~\ref{vanish}, the conclusion follows from Corollary 3.4.7 in Hong--Kang \cite{Hong}.  
\end{proof}

Given a dominant weight $\widehat{\lambda}$ of $\mathfrak{sl}_{k-1}$, we write $\widehat{V}^{\widehat{\lambda}}$ for the irreducible $U_q(\mathfrak{sl}_{k-1})$-representation with highest weight $\widehat{\lambda}$.  

\begin{thm} \label{decomp}
	$V^{\lambda} = \bigoplus_{j=0}^m U_q(\mathfrak{sl}_{k-1}) v_{j}$.  
\end{thm}

\begin{proof}
	By the Pieri rule, $V^{\lambda}$ is isomorphic as a $U_q(\mathfrak{sl}_{k-1})$-representation to $\bigoplus_{j=0}^m \widehat{V}^{\lambda^j}$.  In view of Proposition~\ref{iso}, the conclusion follows from Schur's lemma.  
\end{proof}

\begin{rem}
	What makes the statement of Theorem~\ref{decomp} so simple is that $v_0, \ldots, v_m$ constitute the highest weight vectors in the restriction of $V^{\lambda}$ to $U_q(\mathfrak{sl}_{k-1})$.  The decomposition of $V^{\lambda}$ cannot be so succinctly described if $\lambda$ is not rectangular.  
	
	Indeed, suppose that there exist nonnegative integers $m_{a-1}, \ldots, m_1$, not all zero, such that $\lambda = m \omega_a + m_{a-1} \omega_{a-1} + \cdots + m_1 \omega_1$.  Let $a' < a$ be maximal for which $m_{a'}$ is nonzero.  Then $\tilde{e}_1, \ldots, \tilde{e}_{k-2}$ all annihilate $\tilde{f}_{k-1} \cdots \tilde{f}_{a'} \overline{v_{\lambda}}$, but $f_{k-1} \cdots f_{a'} v_{\lambda}$ is \textit{not} a highest weight vector with respect to $U_q(\mathfrak{sl}_{k-1})$, for 
	\begin{align*}
	& e_{a-1} f_{k-1} \cdots f_{a'} v_{\lambda} = f_{k-1} \cdots f_{a} [q^{h_{a-1}};0]_q f_{a-2} \cdots f_{a'} v_{\lambda}
	\\ & = f_{k-1} \cdots f_{a} [q^{\lambda(h_{a-1}) - \alpha_{a'}(h_{a-1}) - \cdots - \alpha_{a-2}(h_{a-1})};0]_q f_{a-2} \cdots f_{a'} v_{\lambda}
	\\ & = f_{k-1} \cdots f_{a} [q^{m_{a-1} \omega_{a-1}(h_{a-1}) - \delta_{a' \leq a-2} \alpha_{a-2}(h_{a-1})};0]_q f_{a-2} \cdots f_{a'} v_{\lambda}
	\\ & = f_{k-1} \cdots f_{a} [q^{m_{a-1} + 1 - \delta_{a',a-1}};0]_q f_{a-2} \cdots f_{a'} v_{\lambda} \neq 0.
	\end{align*}

\end{rem}

\begin{exam} \label{square}
	Set $k := 3$ and $\lambda := 2 \omega_2$.  Then \[B_{\lambda} = \left \lbrace 
		\begin{smallmatrix} 1 & 1 \\ 2 & 2 \end{smallmatrix},
		\begin{smallmatrix} 1 & 1 \\ 2 & 3 \end{smallmatrix},
		\begin{smallmatrix} 1 & 2 \\ 2 & 3 \end{smallmatrix},
		\begin{smallmatrix} 1 & 1 \\ 3 & 3 \end{smallmatrix},
		\begin{smallmatrix} 1 & 2 \\ 3 & 3 \end{smallmatrix},
		\begin{smallmatrix} 2 & 2 \\ 3 & 3 \end{smallmatrix} \right \rbrace,\]
	and $\overline{v_{\lambda}} = \begin{smallmatrix} 1 & 1 \\ 2 & 2 \end{smallmatrix}$.  
		
	Furthermore, $(\overline{v_0}, \overline{v_1}, \overline{v_2}) = (\begin{smallmatrix} 1 & 1 \\ 2 & 2 \end{smallmatrix},  \begin{smallmatrix} 1 & 1 \\ 2 & 3 \end{smallmatrix},  \begin{smallmatrix} 1 & 1 \\ 3 & 3 \end{smallmatrix})$, and the partition of sets \[B_{\lambda} = \left \lbrace  \begin{smallmatrix} 1 & 1 \\ 2 & 2 \end{smallmatrix} \right \rbrace \sqcup \left \lbrace  \begin{smallmatrix} 1 & 1 \\ 2 & 3 \end{smallmatrix},  \begin{smallmatrix} 1 & 2 \\ 2 & 3 \end{smallmatrix} \right \rbrace \sqcup \left \lbrace  \begin{smallmatrix} 1 & 1 \\ 3 & 3 \end{smallmatrix},  \begin{smallmatrix} 1 & 2 \\ 3 & 3 \end{smallmatrix},  \begin{smallmatrix} 2 & 2 \\ 3 & 3 \end{smallmatrix}\right \rbrace\] underlies the decomposition of $U_q(\mathfrak{sl}_2)$-representations \[V^{\lambda} = U_q(\mathfrak{sl}_2) v_0 \oplus U_q(\mathfrak{sl}_2) v_1 \oplus U_q(\mathfrak{sl}_2) v_2.\]  
\end{exam}

\begin{exam} \label{triang}
	Set $k:=3$ and $\lambda := \omega_2 + \omega_1$.  Then \[B_{\lambda} = \left \lbrace 
	\begin{smallmatrix} 1 & 1 \\ 2 \end{smallmatrix},
	\begin{smallmatrix} 1 & 2 \\ 2 \end{smallmatrix},
	\begin{smallmatrix} 1 & 3 \\ 2 \end{smallmatrix},
	\begin{smallmatrix} 1 & 1 \\ 3 \end{smallmatrix},
	\begin{smallmatrix} 1 & 2 \\ 3 \end{smallmatrix},
	\begin{smallmatrix} 2 & 2 \\ 3 \end{smallmatrix},
	\begin{smallmatrix} 1 & 3 \\ 3 \end{smallmatrix},
	\begin{smallmatrix} 2 & 3 \\ 3 \end{smallmatrix} \right \rbrace,\] and $\overline{v_{\lambda}} = \begin{smallmatrix} 1 & 1 \\ 2 \end{smallmatrix}$.  
	
	The partition of sets \[B_{\lambda} = \left \lbrace 
	\begin{smallmatrix} 1 & 1 \\ 2 \end{smallmatrix},
	\begin{smallmatrix} 1 & 2 \\ 2 \end{smallmatrix} \right \rbrace \sqcup
	\left \lbrace \begin{smallmatrix} 1 & 3 \\ 2 \end{smallmatrix} \right \rbrace \sqcup
	\left \lbrace \begin{smallmatrix} 1 & 1 \\ 3 \end{smallmatrix},
	\begin{smallmatrix} 1 & 2 \\ 3 \end{smallmatrix},
	\begin{smallmatrix} 2 & 2 \\ 3 \end{smallmatrix} \right \rbrace \sqcup
	\left \lbrace \begin{smallmatrix} 1 & 3 \\ 3 \end{smallmatrix},
	\begin{smallmatrix} 2 & 3 \\ 3 \end{smallmatrix} \right \rbrace\] underlies the decomposition of $U_q(\mathfrak{sl}_2)$-representations \[V^{\lambda} \cong \widehat{V}^{\omega_1} \oplus \widehat{V}^{0} \oplus \widehat{V}^{2 \omega_1} \oplus \widehat{V}^{\omega_1}.\]
	
	Whereas $\tilde{e}_1$ vanishes on $\tilde{f}_2 \tilde{f}_1 \overline{v_{\lambda}} = \begin{smallmatrix} 1 & 3 \\ 2 \end{smallmatrix}$, however, $e_1$ does not vanish on $f_2 f_1 v_{\lambda}$, so $f_2 f_1 v_{\lambda}$ is \textnormal{not} a highest weight vector with respect to $U_q(\mathfrak{sl}_2)$.  
\end{exam}

\subsection{Restricting the lower global basis}
For all $0 \leq j \leq m$: 
\begin{itemize}
\item Fix a weight vector $u_j \in \widehat{V}^{\lambda_j}_{\lambda_j}$; 
\item Let $\phi_j \colon \widehat{V}^{\lambda_j} \rightarrow U_q(\mathfrak{sl}_{k-1})v_j$ be the isomorphism with $\phi_j(u_j) = v_j$; 
\item Let $(\widehat{L}^{\lambda^j}, \widehat{B}^{\lambda^j})$ be the crystal basis of $\widehat{V}^{\lambda^j}$ such that $\widehat{L}^{\lambda^j}_{\lambda^j} = A_0 u_j$ and $\widehat{B}^{\lambda^j}_{\lambda^j} = \lbrace \overline{u_j} \rbrace$;   
\item Let $\widehat{G}^{\lambda^j} = \lbrace \widehat{G}^j_b \rbrace_{b \in \widehat{B}^{\lambda^j}}$ be the lower global basis of $\widehat{V}^{\lambda^j}$.  
\end{itemize}

\begin{lem} \label{in}
	The isomorphism $\phi_j$ sends $\widehat{L}^{\lambda^j}$ into $L^{\lambda}$, and the induced map $\phi_j \colon \widehat{L}^{\lambda^j}/q\widehat{L}^{\lambda^j} \rightarrow L^{\lambda}/qL^{\lambda}$ sends $\widehat{B}^{\lambda^j}$ into $B^{\lambda}$.    
\end{lem}

\begin{proof}
	Since $v_j = \tilde{f}_{k-1}^{j} \cdots \tilde{f}_{a}^j v_{\lambda}$, the former claim follows from Proposition~\ref{lattice}.  For the latter claim, it suffices to show that $\tilde{f}_{i_r} \cdots \tilde{f}_{i_1} \overline{u_j} \neq 0$ implies $\tilde{f}_{i_r} \cdots \tilde{f}_{i_1} \overline{v_j} \neq 0$ for all $r \geq 0$ and sequences $i_1, \ldots, i_r \in \lbrace 1, \ldots, k-2 \rbrace$.  
	
	By Theorem~\ref{crystal}, we may substitute $T_{\lambda^j}$ for $\overline{u_j}$ and $T_j := \tilde{f}_{k-1}^j \cdots \tilde{f}_a^j (T_{\lambda})$ for $\overline{v_j}$.  By Lemma~\ref{expl}, the entries of $T_j$ in the rightmost $j$ boxes in the bottom row are all equal to $k$, and the tableau obtained by removing these boxes is $T_{\lambda^j}$.\footnote{As a partition, $\lambda^j$ \textit{always} refers to $(m^{a-1}, m-j)$.  If $a < k-1$, there is no ambiguity, but, if $a = k-1$, then $\omega_a$ vanishes in the weight lattice of $\mathfrak{sl}_{k-1}$, so $\lambda^j$ could mean $(j^{k-2})$.  It does not.  The crystal structures on $SSYT((j^{k-2}), k-1)$ and $SSYT((m^{k-2}, m-j), k-1)$ are identical, and we favor the latter model for its pictorial compatibility with the crystal structure on $SSYT((m^{k-1}), k)$.  }  Inducting on $r$, we see that if $\tilde{f}_{i_r} \cdots \tilde{f}_{i_1} (T_{\lambda^j}) \neq 0$, then $\tilde{f}_{i_r} \cdots \tilde{f}_{i_1}(T_j)$ is the unique tableau such that (i) the entries in the rightmost $j$ boxes in the bottom row are all equal to $k$, and (ii) the tableau obtained by removing these boxes is $\tilde{f}_{i_r} \cdots \tilde{f}_{i_1}(T_{\lambda^j})$.  
\end{proof}

\begin{thm} \label{res}
	The following assertions hold.  
	\begin{enumerate}
		\item $L^{\lambda} = \bigoplus_{j=0}^m \phi_j(\widehat{L}^{\lambda^j})$.
		\item $B^{\lambda} = \bigsqcup_{j=0}^m \phi_j(\widehat{B}^{\lambda^j})$.
		\item $G^{\lambda} = \bigsqcup_{j=0}^m \phi_j(\widehat{G}^{\lambda^j})$.  In particular, $\phi_j(\widehat{G}^j_b) = G^{\lambda}_{\phi_j(b)}$ for all $b \in \widehat{B}^{\lambda^j}$.  
	\end{enumerate} 
\end{thm}  

\begin{proof}		
	We freely apply Lemma~\ref{in}.
	\begin{enumerate}
		\item 	Set $L := \bigoplus_{j=0}^m \widehat{L}^{\lambda^j}$, and define $\phi \colon L \rightarrow L^{\lambda}$ by $\phi := \sum_{j=0}^m \phi_j$.  Note that $\phi$ induces maps $\phi_q \colon \mathbb{Q}(q) \otimes_{A_0} L \rightarrow \mathbb{Q}(q) \otimes_{A_0} L^{\lambda}$ and $\phi_{\text{zero}} \colon L/qL \rightarrow L^{\lambda}/qL^{\lambda}$.  It follows from Theorem~\ref{decomp} that $\phi_q$ is an isomorphism of $\mathbb{Q}(q)$-vector spaces.  Thus, $\operatorname{rank}_{A_0} L = \operatorname{rank}_{A_0} L^{\lambda}$, which implies $\dim_{\mathbb{Q}} L/qL = \dim_{\mathbb{Q}} L^{\lambda}/qL^{\lambda}$.  Since $\phi_{\text{zero}}$ is injective, we see that $\phi_{\text{zero}}$ is an isomorphism of $\mathbb{Q}$-vector spaces, and the conclusion follows from Nakayama's Lemma.  
		
		\item Immediate from the observation that $\phi_{\text{zero}}$ is an isomorphism.  
		
		\item By assertion (1), we see that $\bigsqcup_{j=0}^m \phi_j(\widehat{G}^{\lambda^j})$ is an $A_0$-basis of $L^{\lambda}$.  
		
		Let $b \in \widehat{B}^{\lambda^j}$.  We claim that $\phi_j(\widehat{G}^{j}_b) \in U_A(\mathfrak{sl}_k) v_{\lambda} \cap L^{\lambda} \cap \psi(L^{\lambda})$.  Since $\widehat{G}^j_b \in \widehat{L}^{\lambda^j}$, it follows that $\phi_j(\widehat{G}^j_b) \in L^{\lambda}$.  
		
		Fix $p \in U_A(\mathfrak{sl}_{k-1})$ such that $\widehat{G}^j_b = p u_j$.  Note that $\phi_j(p u_j) = p v_j = p f_{k-1}^{(j)} \cdots f_{a}^{(j)} v_{\lambda} \in U_A(\mathfrak{sl}_k) v_{\lambda}$.  
		
		Furthermore, since $p u_j = \psi(p) u_j$, it follows that $p v_j =  \psi(p) v_j = \psi(p) f_{k-1}^{(j)} \cdots f_{a}^{(j)} v_{\lambda} = \psi(p) \psi(f_{k-1}^{(j)} \cdots f_{a}^{(j)}) v_{\lambda} = \psi(p v_j) \in \psi(L^{\lambda})$.  
		
		Hence $\bigsqcup_{j=0}^m \phi_j(\widehat{G}^{\lambda^j})$ is a $\mathbb{Q}$-basis of $U_A(\mathfrak{sl}_k) v_{\lambda} \cap L^{\lambda} \cap \psi(L^{\lambda})$, for it is a linearly independent subset of cardinality equal to that of $B^{\lambda}$.  
	\end{enumerate}

\end{proof}

\begin{exam} \label{resquare}
	Returning to Example~\ref{square}, we set $k:=3$ and $\lambda:=2 \omega_2$.  Then \[G^{\lambda} = \left \lbrace v_0, v_1, f_1 v_1, v_2, f_1 v_2, f_1^{(2)} v_2 \right \rbrace.\]
	
	Furthermore, $(\lambda^0, \lambda^1, \lambda^2) = (0, \omega_1, 2 \omega_1)$, and \[\left(\widehat{G}^{\lambda^0}, \widehat{G}^{\lambda^1}, \widehat{G}^{\lambda^2}\right) = \left(\lbrace u_0 \rbrace, \lbrace u_1, f_1 u_1 \rbrace, \lbrace u_2, f_1 u_2, f_1^{(2)} u_2 \rbrace \right).\]
\end{exam}

\begin{exam} \label{retriang}
	Returning to Example~\ref{triang}, we set $k:=3$ and $\lambda:=\omega_2 + \omega_1$.  Then \[G^{\lambda} = \left \lbrace v_{\lambda}, f_1 v_{\lambda}, f_2 f_1 v_{\lambda}, f_2 v_{\lambda}, f_1 f_2 v_{\lambda}, f_1^{(2)} f_2 v_{\lambda}, f_2^{(2)} f_1 v_{\lambda}, f_1 f_2^{(2)} f_1 v_{\lambda} \right \rbrace.\]  Since $f_2 f_1 v_{\lambda}$ is not a highest weight vector with respect to $U_q(\mathfrak{sl}_2)$, the global basis is not compatible with the decomposition \[V^{\lambda} \cong \widehat{V}^{\omega_1} \oplus \widehat{V}^{0} \oplus \widehat{V}^{2 \omega_1} \oplus \widehat{V}^{\omega_1}.\]     
\end{exam}

\subsection{Restricting the upper global basis}
For all $0 \leq j \leq m$, let $(\cdot, \cdot)_j$ denote the Shapovalov form on $\widehat{V}^{\lambda^j}$.  Note that $(\cdot, \cdot)_j = \frac{(\phi_j(\cdot), \phi_j(\cdot))}{(v_j, v_j)}$.  

\begin{thm} \label{upper}
	Let $\widehat{F}^{\lambda^j} = \lbrace \widehat{F}^j_b \rbrace_{b \in \widehat{B}^{\lambda^j}}$ be the upper global basis of $\widehat{V}^{\lambda^j}$.  Then $F^{\lambda} = \bigsqcup_{j=0}^m \frac{\phi_j(\widehat{F}^{\lambda^j})}{(v_j, v_j)}$.  In particular, $\frac{\phi_j(\widehat{F}^{j}_b)}{(v_j, v_j)} = F^{\lambda}_{\phi_j(b)}$ for all $b \in \widehat{B}^{\lambda^j}$.  
\end{thm}

\begin{proof}
	Let $b \in \widehat{B}^{\lambda^j}$.  For all $b' \in \widehat{B}^{\lambda^j}$, we see that  \[\frac{(\phi_j(\widehat{F}^j_b),\phi_j(\widehat{G}^j_{b'}))}{(v_j, v_j)} = (\widehat{F}^j_b, \widehat{G}^j_{b'})_j = \delta_{b,b'}.\]
	
	Furthermore, for all $j' \neq j$ and $b' \in \widehat{B}^{\lambda^{j'}}$, it follows from Propositions~\ref{perpen} and \ref{wvec} that $(\phi_j(\widehat{F}^j_b),\phi_{j'}(\widehat{G}^{j'}_{b'})) = 0$, and the conclusion follows from Theorem~\ref{res}.  
\end{proof}

\begin{lem} \label{gauss}
	For all $a \leq a' \leq k-1$, set $v_{a',j} := f_{a'}^{(j)} \cdots f_{a}^{(j)} v_{\lambda}$.  Then $(v_{a',j}, v_{a',j}) = \begin{bsmallmatrix} m \\ j \end{bsmallmatrix}_q$ for all $0 \leq j \leq m$.  
\end{lem}

\begin{proof}
	We induct on $a'$.  For the base case, note that
	\begin{align*}
	(f_a^{(j)}v_{\lambda}, f_a^{(j)} v_{\lambda}) & = (v_{\lambda}, e_a^{(j)} f_a^{(j)} v_{\lambda}) 
	\\ &= \frac{(v_{\lambda},e_a^{(j-1)}f_a^{(j)}e_a v_{\lambda}) + (v_{\lambda}, e_a^{(j-1)} f_a^{(j-1)} [q^{h_a}; 1-j]_q v_{\lambda})}{[j]_q}
	\\ &= \frac{[m+1-j]_q(v_{\lambda}, e_a^{(j-1)} f_a^{(j-1)} v_{\lambda})}{[j]_q},
	\end{align*}
	so we see by induction on $j$ that 
	\[(f_a^{(j)} v_{\lambda}, f_a^{(j)} v_{\lambda}) = \frac{[m+1-j]_q \cdots [m]_q}{[j]_q \cdots [1]_q} (v_{\lambda}, v_{\lambda}) = \begin{bmatrix} m \\ j \end{bmatrix}_q.\]
	
	For the inductive step, note that 
	\begin{align*}
	& (f_{a'}^{(j')} v_{a'-1,j}, f_{a'}^{(j')} v_{a'-1,j}) 
	= (v_{a'-1,j}, e_{a'}^{(j')} f_{a'}^{(j')} v_{a'-1,j})
	\\ & = \frac{(v_{a'-1,j}, e_{a'}^{(j'-1)} f_{a'}^{(j')} e_{a'} v_{a'-1,j}) + (v_{a'-1,j}, e_{a'}^{(j'-1)} f_{a'}^{(j'-1)} [q^{h_{a'}}; 1-j']_q v_{a'-1,j})}{[j']_q}
	\\ & = \frac{[j+1-j']_q (v_{a'-1,j},e_{a'}^{(j'-1)} f_{a'}^{(j'-1)} v_{a'-1,j})}{[j']_q},
	\end{align*}
	so we see by induction on $j'$ that 
	\begin{align*}
	& (f_{a'}^{(j)} v_{a'-1,j}, f_{a'}^{(j)} v_{a'-1,j}) = \frac{[1]_q \cdots [j]_q}{[j]_q \cdots [1]_q} (v_{a'-1,j}, v_{a'-1,j}) = \begin{bmatrix} m \\ j \end{bmatrix}_q.
	\end{align*}
\end{proof}

\begin{prop} \label{homo}
	For all $0 \leq j \leq m$, let $\widehat{\mathsf{V}}^{\lambda^j}$ be the $\mathbb{C}$-form of the $q=1$ specialization of $\widehat{V}^{\lambda^j}$.  Let $\tau_j \colon \widehat{\mathsf{V}}^{\lambda^j} \rightarrow \mathsf{V}^{\lambda}$ be given by $\widehat{\mathsf{F}}^{j}_b \mapsto \mathsf{F}^{\lambda}_{\phi_j(b)}$.  Then $\tau_j$ is an $\mathfrak{sl}_{k-1}$-module homomorphism.  
\end{prop}

\begin{proof}
	It follows from Theorem~\ref{upper} and Lemma~\ref{gauss} that the $U_q(\mathfrak{sl}_{k-1})$-module homomorphism $\phi_j \colon \widehat{V}^{\lambda^j} \rightarrow V^{\lambda}$ is given by $\widehat{F}^j_b \mapsto \begin{bsmallmatrix} m \\ j \end{bsmallmatrix}_q F^{\lambda}_{\phi_j(b)}$.  Hence the $\mathbb{C}$-linear map $\widehat{\mathsf{V}}^{\lambda^j} \rightarrow \mathsf{V}^{\lambda}$ given by $\widehat{\mathsf{F}}^j_b \mapsto \binom{m}{j} \mathsf{F}^{\lambda}_{\phi_j(b)}$ is an $\mathfrak{sl}_{k-1}$-module homomorphism.  
\end{proof}

\section{Proof of Rhoades's Theorem}

Let $\Lambda = (\Lambda_1, \ldots, \Lambda_k)$ be a partition, and let $\mathsf{V}^{\Lambda}$ be the irreducible $GL_k(\mathbb{C})$-representation with highest weight $\Lambda$.  Set $\lambda := \Lambda_1 E_1 + \cdots + \Lambda_k E_k$.  Then $\lambda$ is the image of $\Lambda$ in the weight lattice of $\mathfrak{sl}_k$, and $\mathsf{V}^{\Lambda}$ is isomorphic as an $\mathfrak{sl}_k$-representation to $\mathsf{V}^{\lambda}$.  Thus, we may consider $\mathsf{F}^{\lambda}$ a basis of $\mathsf{V}^{\Lambda}$ (such that the $\mathfrak{sl}_k$-action on $\mathsf{V}^{\Lambda}$ induced from the $GL_k(\mathbb{C})$-module structure agrees with that induced from the $\mathfrak{sl}_k$-action on $\mathsf{V}^{\lambda}$).  

Identify $B^{\lambda}$ with $SSYT(\Lambda, k)$.  

\begin{thm}[Berenstein--Zelevinsky \cite{Berenstein}, Proposition 8.8; Stembridge \cite{Stembridge}, Theorem 3.1] \label{bz}
	Set $\epsilon_{\Lambda} := (-1)^{\sum_{i=1}^k (i-1) \Lambda_i}$.  For all $b \in SSYT(\Lambda, k)$, 
	\[w_{0,k} \cdot \mathsf{F}^{\lambda}_b = \epsilon_{\Lambda} \mathsf{F}^{\lambda}_{\xi_k(b)}.\] 
\end{thm}

\begin{exam} \label{sq}
	Set $\Lambda := (2,2,0)$.  Then $\lambda = 2 \omega_2$ and $\epsilon_{\Lambda} = 1$.  Furthermore, $w_{0,3}$ acts on $\mathsf{V}^{\Lambda}$ by 
	\[
	\mathsf{F}_{\begin{smallmatrix} 1  & 1 \\ 2 & 2 \end{smallmatrix}} \leftrightarrow \mathsf{F}_{\begin{smallmatrix} 2 & 2 \\ 3 & 3 \end{smallmatrix}}; \quad
	\mathsf{F}_{\begin{smallmatrix} 1  & 1 \\ 2 & 3 \end{smallmatrix}} \leftrightarrow \mathsf{F}_{\begin{smallmatrix} 1 & 2 \\ 3 & 3 \end{smallmatrix}}; \quad \mathsf{F}_{\begin{smallmatrix} 1  & 2 \\ 2 & 3 \end{smallmatrix}} \circlearrowleft, \quad \text{and} \quad \mathsf{F}_{\begin{smallmatrix} 1  & 1 \\ 3 & 3 \end{smallmatrix}} \circlearrowleft.\]	
\end{exam}

\begin{exam}
	Set $\Lambda := (2,1,0)$.  Then $\lambda = \omega_2 + \omega_1$ and $\epsilon_{\Lambda} = -1$.  Furthermore, $w_{0,3}$ acts on $\mathsf{V}^{\Lambda}$ by 
	\[	
	\mathsf{F}_{\begin{smallmatrix} 1 & 1 \\ 2 \end{smallmatrix}} \leftrightarrow - \mathsf{F}_{\begin{smallmatrix} 2 & 3 \\ 3 \end{smallmatrix}}; \quad 
	\mathsf{F}_{\begin{smallmatrix} 1 & 2 \\ 2 \end{smallmatrix}} \leftrightarrow - \mathsf{F}_{\begin{smallmatrix} 2 & 2 \\ 3 \end{smallmatrix}}; \quad \mathsf{F}_{\begin{smallmatrix} 1 & 3 \\ 2 \end{smallmatrix}} \leftrightarrow - \mathsf{F}_{\begin{smallmatrix} 1 & 2 \\ 3 \end{smallmatrix}}, \quad \text{and} \quad 
	\mathsf{F}_{\begin{smallmatrix} 1 & 1 \\ 3 \end{smallmatrix}} \leftrightarrow - \mathsf{F}_{\begin{smallmatrix} 1 & 3 \\ 3 \end{smallmatrix}}.  
	\]
\end{exam}

Suppose $\Lambda$ is rectangular.  Let $a$ and $m$ be positive integers for which $\Lambda = (m^a)$.  If $a = k$, then $SSYT(\Lambda, k)$ consists of exactly one tableau, so we may assume $a \leq k-1$.  For all $0 \leq j \leq m$, set $\Lambda^j := (m^{a-1},m-j)$, and let $\widehat{\mathsf{V}}^{\Lambda^j}$ be the irreducible $GL_{k-1}(\mathbb{C})$-representation with highest weight $\Lambda^j$.  

\begin{prop} \label{ghom}
	The map $\tau_j \colon \widehat{\mathsf{V}}^{\Lambda^j} \rightarrow \mathsf{V}^{\Lambda}$ given by $\widehat{\mathsf{F}}^{j}_b \mapsto \mathsf{F}^{\lambda}_{\phi_j(b)}$ is a $GL_{k-1}(\mathbb{C})$-module homomorphism.  
\end{prop}

\begin{proof}
By the Pieri rule, $\mathsf{V}^{\Lambda} \cong \bigoplus_{j=0}^m \widehat{\mathsf{V}}^{\Lambda^j}$, so there exist $GL_{k-1}(\mathbb{C})$-module homomorphisms $\tau'_j \colon \widehat{\mathsf{V}}^{\Lambda^j} \rightarrow \mathsf{V}^{\Lambda}$ such that $\mathsf{V}^{\Lambda} = \bigoplus_{j=0}^m \tau'_j(\widehat{\mathsf{V}}^{\Lambda^j})$.  Since $\tau'_j$ and $\tau_j$ are $\mathfrak{sl}_{k-1}$-module homomorphisms (cf. Proposition~\ref{homo}), it follows that $\tau'_j$ agrees with $\tau_j$ up to scaling by a constant in $\mathbb{C}$.  
\end{proof}

Identify $\widehat{B}^{\lambda^j}$ with $SSYT(\Lambda^j, k-1)$ for all $j$.  Given $b \in SSYT(\Lambda^j, k-1)$, note that $\phi_j(b) \in SSYT(\Lambda, k)$ is the unique tableau such that (i) the entries in the rightmost $j$ boxes in the bottom row are all equal to $k$, and (ii) the tableau obtained by removing these boxes is $b$ (cf. Lemma~\ref{in}).  It follows that $\phi_j$ commutes with $\xi_{k-1}$.  

Recall from Theorem~\ref{res} that $SSYT(\Lambda, k) = \bigsqcup_{j=0}^m \phi_j(SSYT(\Lambda^j, k-1))$.  Thus, it suffices to describe the action of $c_k$ on the basis elements in $\mathsf{V}^{\Lambda}$ corresponding to tableaux in $\phi_j(SSYT(\Lambda^j, k-1))$, whence Rhoades's cyclic sieving result follows.  

\begin{thm}[Rhoades \cite{Rhoades}, Proposition 5.5] \label{main}
	For all $b \in SSYT(\Lambda^j, k-1)$,
	\[c_k \cdot \mathsf{F}^{\lambda}_{\phi_j(b)} = (-1)^{(a-1)j} \mathsf{F}^{\lambda}_{J(\phi_j(b))}.\]
\end{thm}

\begin{proof}
	Note that $\xi_{k} \circ \xi_{k-1} = J$.  Invoking Theorem~\ref{bz} and Proposition~\ref{ghom}, we find
\begin{align*}
c_k \cdot \mathsf{F}^{\lambda}_{\phi_j(b)} & = w_{0,k} w_{0, k-1} \cdot \mathsf{F}^{\lambda}_{\phi_j(b)} = w_{0,k} \cdot \tau_j \left(w_{0,k-1} \cdot \widehat{\mathsf{F}}^{j}_b\right)
\\ & = w_{0,k} \cdot \tau_j\left(\epsilon_{\Lambda^j} \widehat{\mathsf{F}}^{j}_{\xi_{k-1}(b)}\right) = w_{0,k} \cdot \epsilon_{\Lambda^j} \mathsf{F}^{\lambda}_{\phi_j(\xi_{k-1}(b))}
\\ & = w_{0,k} \cdot \epsilon_{\Lambda^j} \mathsf{F}^{\lambda}_{\xi_{k-1}(\phi_j(b))} = \epsilon_{\Lambda} \epsilon_{\Lambda^j} \mathsf{F}^{\lambda}_{\xi_k(\xi_{k-1}(\phi_j(b)))}
\\ & = (-1)^{(a-1)j} \mathsf{F}^{\lambda}_{J(\phi_j(b))}.
\end{align*}
\end{proof}

\begin{exam}
	Set $\Lambda := (2,2,0)$.  Then \[\widehat{B}^{\lambda^0} = \left \lbrace \begin{smallmatrix} 1 & 1 \\ 2 & 2 \end{smallmatrix} \right \rbrace; \quad \widehat{B}^{\lambda^1} = \left \lbrace  \begin{smallmatrix} 1 & 1 \\ 2  \end{smallmatrix},  \begin{smallmatrix} 1 & 2 \\ 2  \end{smallmatrix} \right \rbrace, \quad \text{and} \quad \widehat{B}^{\lambda^2} = \left \lbrace  \begin{smallmatrix} 1 & 1 \\ &  \end{smallmatrix},  \begin{smallmatrix} 1 & 2 \\ & \end{smallmatrix},  \begin{smallmatrix} 2 & 2 \\ & \end{smallmatrix}\right \rbrace.\]  
	
	Furthermore, 
	\begin{align*}
	& \phi_0(\begin{smallmatrix} 1 & 1 \\ 2 & 2 \end{smallmatrix}) = (\begin{smallmatrix} 1 & 1 \\ 2 & 2 \end{smallmatrix}); \\ & \phi_1(\begin{smallmatrix} 1 & 1 \\ 2  \end{smallmatrix},  \begin{smallmatrix} 1 & 2 \\ 2  \end{smallmatrix}) = (\begin{smallmatrix} 1 & 1 \\ 2 & 3 \end{smallmatrix}, \begin{smallmatrix} 1 & 2 \\ 2 & 3 \end{smallmatrix}), \quad \text{and} \\ & \phi_2(\begin{smallmatrix} 1 & 1 \\ &  \end{smallmatrix},  \begin{smallmatrix} 1 & 2 \\ & \end{smallmatrix},  \begin{smallmatrix} 2 & 2 \\ & \end{smallmatrix}) = (\begin{smallmatrix} 1 & 1 \\ 3 & 3  \end{smallmatrix},  \begin{smallmatrix} 1 & 2 \\ 3 & 3 \end{smallmatrix},  \begin{smallmatrix} 2 & 2 \\ 3 & 3 \end{smallmatrix}).
	\end{align*}
	
	From Theorem~\ref{bz}, we see that $w_{0,2}$ acts on $\widehat{\mathsf{V}}^{\Lambda^0} \oplus \widehat{\mathsf{V}}^{\Lambda^1} \oplus \widehat{\mathsf{V}}^{\Lambda^2}$ by \[\widehat{\mathsf{F}}^0_{\begin{smallmatrix} 1 & 1 \\ 2 & 2 \end{smallmatrix}} \circlearrowleft; \quad  \widehat{\mathsf{F}}^1_{\begin{smallmatrix} 1 & 1 \\ 2  \end{smallmatrix}} \leftrightarrow - \widehat{\mathsf{F}}^1_{\begin{smallmatrix} 1 & 2 \\ 2  \end{smallmatrix}}; \quad \widehat{\mathsf{F}}^2_{\begin{smallmatrix} 1 & 1 \\ &  \end{smallmatrix}} \leftrightarrow \widehat{\mathsf{F}}^2_{\begin{smallmatrix} 2 & 2 \\ & \end{smallmatrix}}, \quad \text{and} \quad \widehat{\mathsf{F}}^2_{\begin{smallmatrix} 1 & 2 \\ & \end{smallmatrix}} \circlearrowleft.\]
	
	Thus, $w_{0,2}$ acts on $\mathsf{V}^{\Lambda}$ by \[\mathsf{F}_{\begin{smallmatrix} 1 & 1 \\ 2 & 2 \end{smallmatrix}} \circlearrowleft; \quad  \mathsf{F}_{\begin{smallmatrix} 1 & 1 \\ 2 & 3 \end{smallmatrix}} \leftrightarrow - \mathsf{F}_{\begin{smallmatrix} 1 & 2 \\ 2 & 3 \end{smallmatrix}}; \quad \mathsf{F}_{\begin{smallmatrix} 1 & 1 \\ 3 & 3 \end{smallmatrix}} \leftrightarrow \mathsf{F}_{\begin{smallmatrix} 2 & 2 \\ 3 & 3 \end{smallmatrix}}, \quad \text{and} \quad \mathsf{F}_{\begin{smallmatrix} 1 & 2 \\ 3 & 3 \end{smallmatrix}} \circlearrowleft.\]
	
	Combining this result with that in Example~\ref{sq}, we conclude that $c_3$ acts on $\mathsf{V}^{\Lambda}$ by 
	\[
	\mathsf{F}_{\begin{smallmatrix} 1 & 1 \\ 2 & 2 \end{smallmatrix}} \mapsto \mathsf{F}_{\begin{smallmatrix} 2 & 2 \\ 3 & 3 \end{smallmatrix}} \mapsto
	\mathsf{F}_{\begin{smallmatrix} 1 & 1 \\ 3 & 3 \end{smallmatrix}} \mapsto 
	\mathsf{F}_{\begin{smallmatrix} 1 & 1 \\ 2 & 2 \end{smallmatrix}} \quad \text{and} \quad \mathsf{F}_{\begin{smallmatrix} 1 & 1 \\ 2 & 3 \end{smallmatrix}} \mapsto - \mathsf{F}_{\begin{smallmatrix} 1 & 2 \\ 2 & 3 \end{smallmatrix}} \mapsto \mathsf{F}_{\begin{smallmatrix} 1 & 2 \\ 3 & 3 \end{smallmatrix}} \mapsto \mathsf{F}_{\begin{smallmatrix} 1 & 1 \\ 2 & 3 \end{smallmatrix}},\] which is in agreement with Theorem~\ref{main}.  
	
\end{exam}

\begin{exam}
	Set $\Lambda := (2,1,0)$.  Then \[\mathsf{V}^{\Lambda} \cong \widehat{\mathsf{V}}^{(2,1)} \oplus \widehat{\mathsf{V}}^{(1,1)} \oplus \widehat{\mathsf{V}}^{(2,0)} \oplus \widehat{\mathsf{V}}^{(1,0)}.\]  
	
	Although Theorem~\ref{bz} describes the action of $w_{0,2}$ on the right-hand side, we do not thereby obtain a description of the action of $w_{0,2}$ on the upper global basis of $\mathsf{V}^{\Lambda}$, for Proposition~\ref{ghom} does not apply.  
	
	Furthermore, the action of $c_3$ on $\mathsf{F}^{\lambda}$ does not lift the action of $J$ on $B^{\lambda}$.  Indeed, since $J$ exchanges $\begin{smallmatrix} 1 & 3 \\ 2 \end{smallmatrix}$ and  $\begin{smallmatrix} 1 & 2 \\ 3 \end{smallmatrix}$, the order of $J$ is $6$, not $3$.  
\end{exam}

\section{Acknowledgments}
The author thanks Brendon Rhoades for giving us all something interesting to think about these past years.  He also thanks the anonymous referee for numerous editorial suggestions.

\end{document}